\documentclass[reqno]{amsart}

\usepackage[bookmarks=true,bookmarksnumbered=true,bookmarkstype=toc]{hyperref}
\usepackage{mathtools,amssymb,amsthm}
\mathtoolsset{showonlyrefs,showmanualtags}
\usepackage[abbrev]{amsrefs}
\usepackage{tikz-cd}
\usepackage[shortlabels]{enumitem}
\setlist[enumerate,1]{label=\textup{(\roman*)}}
\AtBeginDocument{
	\def\MR#1{}
}

\theoremstyle{plain}
\newtheorem{thm}{Theorem}[section]
\newtheorem{lem}{Lemma}[section]

\newtheorem{prp}{Proposition}[section]
\theoremstyle{definition}
\newtheorem{defn}{Definition}[section]
\theoremstyle{remark}
\newtheorem{rem}{Remark}[section]
\newtheorem{exmp}{Example}

\newcommand{\calF}{\mathcal{F}}
\newcommand{\calH}{\mathcal{H}}
\newcommand{\calK}{\mathcal{K}}
\newcommand{\calW}{\mathcal{W}}
\newcommand{\N}{\mathbb{N}}
\newcommand{\Z}{\mathbb{Z}}
\newcommand{\C}{\mathbb{C}}
\newcommand{\R}{\mathbb{R}}
\newcommand{\inner}[1]{\langle#1\rangle}
\newcommand{\ol}{\overline}
\newcommand{\fa}{\forall}
\newcommand{\imp}{\implies}
\newcommand{\what}{\widehat}
\newcommand{\del}{\delta}
\newcommand{\oset}{\overset}
\newcommand{\hoi}[1]{[#1)}
\newcommand{\es}{\emptyset}
\newcommand{\dlim}{\lim\limits}
\newcommand{\bd}{\partial}
\newcommand{\ninf}{{n\to\infty}}
\newcommand{\blt}{\ensuremath{\bullet}}
\DeclareMathOperator{\ran}{ran}
\DeclareMathOperator{\cspan}{\overline{span}}
\DeclareMathOperator{\Span}{span\,}
\DeclareMathOperator{\med}{med}

\begin{document}

	\title{Inverses of integral transforms of RKHSs}
	\author{Akira Yamada}
	\address{Tokyo Gakugei University}
	\email{yamada@u-gakugei.ac.jp}
	\thanks{}
	\subjclass[2020]{Primary~46E22, Secondary~42A38}
	\keywords{integral transform, inverse transform, RKHS}

	\begin{abstract}
		The Fourier transform and its inverse are well-known to have complex conjugate
		integral kernels.
		S.~Saitoh demonstrated that this relationship extends to the theory of integral transforms
		of Hilbert spaces of functions under certain conditions.
		In this paper, we derive a necessary and sufficient condition for the inverse
		of an integral transform of a Hilbert space of functions to be represented by
		a complex conjugate integral kernel.
		As an application, we present an alternative proof of Plancherel's theorem
		using the theory of reproducing kernels.
	\end{abstract}

	\maketitle

\section{Introduction}

The Fourier transform $\calF$ and its inverse transform $\calF^*$ for $f \in L^1(\R)$
are given by
\begin{align}
	\calF(f)(x)
	&=\frac{1}{\sqrt{2\pi}}\int_\R f(t)e^{-itx} \, dt, \\
	\calF^*(f)(t)
	&=\frac{1}{\sqrt{2\pi}}\int_\R f(x)e^{itx}\, dx,
\end{align}
Note that the integral kernel of the Fourier transform, $e^{-itx}$, and
the integral kernel of the inverse Fourier transform, $e^{itx}$,
are complex conjugates of each other.
Integral transforms in $L^p$ spaces, when abstracted within the framework of Hilbert spaces,
become the integral transforms of Hilbert spaces, which will be discussed later.
In this case, the image space of the integral transform becomes a
reproducing kernel Hilbert space (RKHS),
and it is well known that its reproducing kernel is expressed
by the inner product (Theorem~\ref{th:int-trans}).
In 1982, S.~Saitoh showed in~\cite{Saitoh82} that in the integral transform of Hilbert spaces,
under certain conditions, the integral kernel of the inverse transform becomes the complex
conjugate of the integral kernel of the original transform,
similar to the case of the Fourier transform described above
(Theorem~\ref{th:inttrf0}, Equation~\eqref{eq:invtrans}).

This paper establishes a close relationship between RKHSs and
the inverse transform formula~\eqref{eq:invtrans} (Theorem~\ref{thm:invtrans}).
We derive a generalized inverse formula~\eqref{eq:invtrans2} as an extension
of~\eqref{eq:invtrans} and apply it to provide an alternative proof of Plancherel's theorem.
Throughout this paper, all operators are linear unless otherwise stated.

\section{Integral transforms of Hilbert spaces}

A Hilbert space $\calH$ of functions on a set $E$ is called an {\em RKHS} if the operator
\[
f\in\calH\mapsto f(x)\in\C
\]
is bounded for any $f\in\calH$ and point $x\in E$.
Then, by Riesz's representation theorem, there exists a function $k_x\in\calH$
such that $f(x)=\inner{f,k_x}$.
This function $k_x$ is called the {\em reproducing kernel} of $\calH$
for the point $x$, and the two-variable function $k(x,y)=\inner{k_y,k_x}$
is called the reproducing kernel of $\calH$.
For more details on RKHS, refer to \cites{Aronszajn50, PaulsenRaghupathi16}.
For the theory and applications of integral transforms of Hilbert spaces, see \cites{Saitoh97, SaitohSawano16} and the references therein.

\begin{defn}
	Let $\calH$ be a Hilbert space and let $V$ be a vector space.
	For a linear map $A\colon\calH\to V$ with closed kernel,
	there exists a unique Hilbert space structure on
	$\ran A=A(\calH)$ of $V$ so that the linear map
	$A\colon\calH\to\ran A$ is a coisometry.
	In other words, we define an inner product on $\ran A$ such that,
	for $f,g\in(\ker A)^\perp$,
	\[
		\inner{Af,Ag}_{\ran A}=\inner{f,g}_\calH.
	\]
	This Hilbert space is called the {\em operator range}
	of the linear map $A$ (cf.~\cite{Sarason94})
\end{defn}

Let $\phi\colon E\to\calH$ be a map from a set $E$ to a Hilbert space $\calH$.
Define an operator $\hat\phi\colon\calH\to\C^E$ by, for $f\in\calH$,
\[\label{sec:int_trans}
(\hat\phi f)(x)=\inner{f,\phi(x)}_\calH,\quad x\in E.
\]
Note that $\ker\hat\phi=\phi(E)^\perp$ is closed in $\calH$.
The operator range $\hat\phi(\calH)$ of the map $\hat\phi$ is a RKHS on $E$,
and the operator $\hat\phi\colon\calH\to\hat\phi(\calH)$ is called
the {\em integral transform} of $\calH$ induced by the map $\phi$.
If the map $\phi$ is clear from the context, we write with $\hat f=\hat\phi f$
and $\hat\calH=\hat\phi(\calH)$.

\begin{thm}[\cite{Saitoh97}*{pp.~20--23}] \label{th:int-trans}
	Under the above settings, the following hold:
	\begin{enumerate}
		\item For any $f\in\calH$, we have $\|\hat f\|_{\hat\calH}\le\|f\|_\calH$.
				Equality holds if and only if $f\in(\ker\hat\phi)^\perp=\cspan\phi(E)$,
				where $\cspan\phi(E)$ denotes the closed linear span of the set $\phi(E)$.
				In particular, $\hat\phi\colon\calH\to\hat\calH$ is an isometry if and only if
				$\phi(E)$ is complete in $\calH$, i.e., the linear span of $\phi(E)$ is dense
				in $\calH$.
		\item The integral transform $\hat\calH$ of $\calH$ is a RKHS on $E$,
				and its reproducing kernel for the point $y\in E$ is the function $\what{\phi(y)}$.
				Moreover, the reproducing kernel of $\hat\calH$ is given by
				\begin{equation}\label{eq:phi-kernel}
					\inner{\phi(y),\phi(x)}_\calH.
				\end{equation}
	\end{enumerate}
\end{thm}

The following is a prototype of the inversion theorem~\cite{Saitoh82}*{Theorem 3.2}
for the integral transforms of Hilbert spaces.

\begin{thm}[\cite{Saitoh82}*{Theorem 3.1}] \label{th:inttrf0}
	Let $\phi\colon E\to\calH$ be a map from a set $E$ to a Hilbert space $\calH$
	of functions on a set $T$,
	and let $\hat\calH$ be the integral transform of $\calH$ induced by $\phi$.
	Then, if all of the conditions
  \begin{enumerate}
		\item $\phi(E)$is complete in $\calH$,
    \item $\ol{\phi(t,\cdot)}\in\hat\calH$, for any $t\in T$,
    \item $\inner{\hat f,\ol{\phi(t,\cdot)}}_{\hat\calH}\in\calH$,
					for any $f\in\calH$ and $t\in T$,
    \item $\inner{\hat f(x),\inner{\phi(y),\phi(x)}_\calH}_{\hat\calH}
     		=\inner{\inner{\hat f(x),\ol{\phi(t,x)}}_{\hat\calH},\phi(t,y)}_\calH$,
				for any $f\in\calH$ and $x,y\in E$,
  \end{enumerate}
  are satisfied, the inverse of the integral transform $\hat\phi$
	is given by
  \begin{equation}\label{eq:invtrans}
    f(t)=\inner{\hat f,\ol{\phi(t,\cdot)}}_{\hat\calH},\quad t\in T.
  \end{equation}
\end{thm}

Although the hypotheses are somewhat restrictive,
Theorem\ref{th:inttrf0} can be regarded as a generalization of the inverse
Fourier transform outlined in Introduction.

\section{Main Results}

The following is our main result concerning the inverse transform.

\begin{thm}[Main Theorem]\label{thm:invtrans}
	Let $E$ be a set, and let $\calH$ be a Hilbert space of functions defined on a set $T$.
	Denote by $\hat\calH$ the integral transform of $\calH$ induced by
	a map $\phi\colon E\to\calH$,
	and let $\phi(t,x)=\phi(x)(t)$ for $(t,x)\in T\times E$.
	Then, the following are equivalent:
	\begin{enumerate}
		\item For every $t \in T$ we have $\ol{\phi(t,\cdot)}\in\hat\calH$,
				and the inverse transform formula~\eqref{eq:invtrans} holds for all
				$f\in\cspan\phi(E)$.
		\item The closed subspace $\cspan\phi(E)$ of $\calH$ is a RKHS on $T$.
	\end{enumerate}
	In particular, if $\calH$ is an RKHS on $T$ and $\phi(E)$ is complete in $\calH$,
	then the inverse transform formula \eqref{eq:invtrans} holds for any $f \in \calH$.
\end{thm}

\begin{proof}
	(i)$\imp$(ii)
	By Theorem~\ref{th:int-trans}, the integral transform is a contraction.
	Thus, from the identity~\eqref{eq:invtrans} and Schwarz's inequality,
	for $f \in \cspan\phi(E)$ and $t \in T$, we have
	\[
		|f(t)|\le\|\ol{\phi(t,\cdot)}\|_{\hat\calH}\|\hat f\|_{\hat\calH}
		\le\|\ol{\phi(t,\cdot)}\|_{\hat\calH}\|f\|_\calH.
	\]
	Therefore, $\cspan\phi(E)$ is a RKHS on $T$.

	(ii)$\imp$(i)
	Let $k_t$ denote the reproducing kernel of the RKHS $\cspan \phi(E)$ for the point $t \in T$. Since $k_t \in \calH$, for any $x \in E$, we have
	\[
	\ol{\phi(t,x)}=\ol{\inner{\phi(x),k_t}}_\calH=\inner{k_t,\phi(x)}_\calH
	=\hat{k}_t(x).
	\]
	Thus, $\ol{\phi(t,\cdot)} = \hat k_t \in \hat\calH$.
	Moreover, since the integral transform $\hat\phi$ is an isometric isomorphism on
	$\cspan\phi(E)=(\ker\hat\phi)^\perp$, for any $f\in\cspan\phi(E)$, we have
	\[
		f(t)=\inner{f,k_t}_\calH=\inner{\hat f,\hat k_t}_{\hat\calH}
		=\inner{\hat f,\ol{\phi(t,\cdot)}}_{\hat\calH}.
	\]
	Therefore,~\eqref{eq:invtrans} holds.
	The final assertion follows because the completeness of $\phi(E)$ in $\calH$ is
	equivalent to $\cspan\phi(E)=\calH$.
\end{proof}

To apply Theorem~\ref{thm:invtrans}, it is necessary to know explicitly
the inner product of the image $\hat\calH$ of the integral transform.
To this end, we prepare the following lemma, which provides a sufficient condition
for the norm of an RKHS to be expressed in terms of the norm of another space.

\begin{lem}\label{lem:iso-met}
	Let $H$ be an RKHS on a set $E$ with reproducing kernel $k$,
	and let $F$ be a subset of $E$.
	Denote by $k_y=k(\cdot, y)$ the reproducing kernel for a point $y \in E$.
	Assume the following conditions:
	\begin{enumerate}
		\item If $f\in H$ and $f|_F=0$, then $f=0$, i.e., $F$ is a uniqueness set
				for $H$.
		\item There exists a map $T\colon\{k_x\}_{x \in F}\to K$ to a Hilbert space
				$K$ such that, for all $x, y \in F$,
				\begin{equation}\label{eq:k(x,y)iso}
					k(x,y)=\inner{Tk_y,Tk_x}_K.
				\end{equation}
	\end{enumerate}
	Then, $T$ extends uniquely to an isometry $\tilde T$ from $H$ to $K$,
	and $H$ is isometrically isomorphic to a closed subspace of $K$.
	In particular, if $K=L^2(F, d\mu)$ and $T$ is defined by the restriction to $F$,
	then the extended isometry $\tilde T$ coincides with the restriction of functions
	in $H$ to $F$.
\end{lem}
	
\begin{proof}
	The space $H|_F$ of restrictions of functions in $H$ to $F$ is an RKHS on $F$,
	with the restriction of $k$ to $F \times F$ as its reproducing kernel.
	This RKHS structure is induced by the operator range of the restriction operator
	$\rho\colon H\to H|_F$.
	Since $\rho$ is injective by (i), $\rho\colon H\to H|_F$ is an isometry
	by the definition of the operator range.
	Moreover, from (ii), for $c_i \in \mathbb{C}$ and $x_i \in F$, we have
	\[
	\Bigl\|\sum_ic_ik_{x_i}|_F\Bigr\|_{H|_F}^2=\sum_{i,j}c_i\ol c_jk(x_j,x_i)
	=\sum_{i,j}c_i\ol c_j\inner{Tk_{x_i},Tk_{x_j}}_K=\Bigl\|\sum_ic_iTk_{x_i}\Bigr\|_K^2.
	\]
	This identity implies that if $\sum_ic_ik_{x_i}|F=0$, then $\sum_ic_iTk{x_i}=0$.
	Thus, the mapping
	\[
		\tilde T\colon\sum_ic_ik_{x_i}|_F\in H|_F\mapsto\sum_ic_iTk_{x_i}\in K
	\]
	is a well-defined linear isometry extending the map $T$.
	By (i), the set $\{k_x\}_{x \in F}$ is complete in $H$,
	which means that the set $\Span\{k_x\}_{x\in F}$ is dense in $H$.
	Consequently, the composition $\tilde T\rho\colon\Span\{k_x\}_{x\in F}\to K$
	extends uniquely to an isometry from $H$ to $K$.
	Thus, $H$ is isometrically isomorphic to the closed subspace of $K$.
	The last assertion can be seen from the fact that, since $H$ is an RKHS,
	strong convergence implies pointwise convergence,
	and in $L^p$ spaces mean convergence implies the existence of a subsequence
	that converges almost everywhere.
\end{proof}
	
The inverse transform formula~\eqref{eq:invtrans} holds for
the integral transform of an RKHS.
Even for a general Hilbert space, a similar formula can be obtained
if there exists an RKHS which is isometric to it.
For convenience we make the following definition.

\begin{defn}
	Let $\calH$ be a Hilbert space and let $\phi$ be a map $\phi\colon E\to\calH$.
	If $\calW$ is an RKHS on a set $T$, and if $S\colon\calH\to\calW$ is an isometry,
	then the composite mapping $E\oset{\phi}{\to}\calH\oset{S}{\to}\calW$
	is called a	{\em transformation sequence} on $T$.
\end{defn}

\begin{thm}[Inverse Transform Theorem]\label{thm:invtrans2}
	Let $E\oset{\phi}{\to}\calH\oset{S}{\to}\calW$ be a transformation sequence on $T$.
	Let $\hat\calH$ denote the integral transform of $\calH$ induced by $\phi$.
	Then, for all $f\in\cspan\phi(E)$ and $t\in T$, the following identity holds:
	\begin{equation}\label{eq:invtrans2}
		(Sf)(t) = \inner{\hat f, \ol{(S\phi)(t, \cdot)}}_{\hat\calH}.
	\end{equation}
\end{thm}
	
\begin{proof}
	Both $\hat\calH = \hat\phi(\calH)$ and $\what{S\phi}(\calW)$ are RKHSs on $E$, and their reproducing kernels coincide since $S$ is an isometry.
	Thus, $\hat\calH = \what{S\phi}(\calW)$.
	Since an isometry is a closed map, so $S(\cspan\phi(E)) = \cspan S(\phi(E))$.
	By applying Theorem~\ref{thm:invtrans} to the integral transform induced by
	$S\phi\colon E\to\calW$, for all $f\in\cspan\phi(E)$, we have
	\[
	(Sf)(t)=\inner{\what{S\phi}(Sf),\ol{(S\phi)(t,\cdot)}}_{\hat\calH}.
	\]
	Since $S$ is an isometry, $\what{S\phi}(Sf)=\hat f$.
	This proves the identity~\eqref{eq:invtrans2}.
\end{proof}

Next we give a simple application of the inverse transform theorem.

\begin{exmp}[cf.~\cite{SaitohSawano16}*{pp.~116,~153}]
	Let $\calH$ be a separable Hilbert space,	and let $\{g_n\}_{n=1}^\infty$
	be a complete orthonormal system for $\calH$.
	The Fourier coefficients of $f\in\calH$	with respect to $\{g_n\}$
	are defined as the sequence $(\hat f_n)\in\C^\N$, where
	\[
		\hat f_n=\inner{f,g_n}_\calH.
	\]
	Consider the integral transform $\hat\calH$ of $\calH$ induced by the map
	$\psi\colon E\to\calH$.
	By Parseval's identity, the map
	\[
		S\colon f\in\calH\mapsto(\hat f_n)\in\ell^2
	\]
	is an isometry.
	Since $\ell^2$ is an RKHS on $\N$ with the reproducing kernel $\del_{ij}$ (Kronecker delta),
	the sequence $\N\oset{\psi}\to\calH\oset{S}\to\ell^2$ is a transformation sequence on $\N$.
	From the inverse transform formula~\eqref{eq:invtrans2},
	for all $f\in\cspan\psi(E)$ and $n\in\N$, we have
	\begin{align}
		\hat f_n
		&=\inner{\hat f,\ol{\what{\psi(\cdot)}n}}_{\hat\calH}
			=\inner{\hat f(x),\ol{\inner{\psi(x),g_n}}\calH}_{\hat\calH}. \label{eq:FourHilb}
	\end{align}
	Expressing the inverse transform as a Fourier series, we have that
	for $f\in\cspan\psi(E)$, the following holds in the sense of strong convergence in $\calH$:
	\[
	f=\sum_{n=0}^\infty\inner{\hat f,\ol{\what{\psi(\cdot)}_n}}_{\hat\calH}g_n.
	\]
\end{exmp}

\subsection{RKHSs isometric to One-Dimensional $L^2$ spaces}

To apply the inverse transform theorem, we need to identify suitable RKHSs that
are isometric to Hilbert spaces.
Here, we focus on RKHSs isometric to $L^2$ spaces on intervals of $\R$,
as these are among the most important Hilbert spaces.

\begin{defn} \label{def:ext intval}
	The notation $(a,b)$ represents the set of real numbers strictly between
	the two extended real numbers	$a$ and $b$ in $\ol\R=\R\cup\{\pm\infty\}$, i.e.,
	\[
		(a,b)=\begin{cases}
			\{x\in\R\colon a<x<b\}, & (a<b), \\
			\es, & (a=b), \\
			\{x\in\R\colon b<x<a\}, & (b<a).
		\end{cases}
	\]
	The notations for half-open intervals $\hoi{a,b}$, closed intervals $[a,b]$, and so forth,
	are defined analogously.
\end{defn}

Under this definition, note that for $c,x,y\in\ol\R$, the following identity holds:
\[
	(c,x)\cap(c,y)=(c,\med\{x,y,c\}),
\]
where $\med\{x,y,z\}$ denotes the {\em median} of $x,y,z\in\ol\R$.

\begin{defn} \label{def:H_rho,c}
	Let $I=(a,b)$, with $-\infty\le a<b\le\infty$, be an open interval in $\ol\R$,
	and let $c\in[a,b]$.
	Consider the space of absolutely continuous complex-valued functions
	$f\in AC(I)$ on the interval $I$ such that $f(c)=0$ and $f'\in L_\rho^2(I)$,
	equipped with the inner product
	\[
	\inner{f,g}=\int_a^bf'(t)\ol{g'(t)}\,\rho(t)\,dt,
	\]
	which is called the first-order {\em Sobolev-type space} on $I$
	and is denoted by $H_{c,\rho}(I)$.
	Here, the weight function $\rho$ is assumed to be positive and measurable
	on $I$ and satisfies the following condition:
	\begin{equation}\label{eq:rho^-1}
		\frac{\chi_{(c,x)}}{\rho}\in L^1((c,x)),\quad\fa x\in I.
	\end{equation}
\end{defn}

\begin{rem}
	In the above definition, when $c$ is an endpoint of $I$, i.e.,
	when $c=a$ or $c=b$, special attention is needed regarding the value of $f(c)$.
	For any $x,y\in I$ with $x<y$, by the Cauchy-Schwarz inequality, we have
	\[
		|f(x)-f(y)|^2=\Bigl|\int_x^yf'(t)\,dt\Bigr|^2
		\le\int_x^y|f'|^2\rho\,dt\int_x^y\frac{dt}{\rho}
		\le\|f\|^2\int_x^y\frac{dt}{\rho}.
	\]
	By assumption~\eqref{eq:rho^-1},
	the right-hand side becomes arbitrarily small as $x,y\to c$.
	Thus, by Cauchy's convergence criterion, for any $f \in H_{c,\rho}(I)$,
	the limit $\dlim_{x\to c} f(x)$ exists.
	This limit value is defined as $f(c)$.
	Furthermore, from~\eqref{eq:rho^-1}, if $\rho\equiv1$, then $c$ must be in $\R$.
\end{rem}

The following theorem is a slight generalization of the RKHSs discussed in
\cite{Saitoh84a} and~\cite{Saitoh88}*{p.~75}.
We will show that the space $H_{c,\rho}(I)$ is an RKHS isometric to $L_\rho^2(I)$.

\begin{thm}\label{th:Sob(I)}
	$H_{c,\rho}(I)$ is an RKHS on $I\cup\{c\}$, and the following hold:
	\begin{enumerate}
		\item The reproducing kernel of $H_{c,\rho}(I)$ is given by
				\[
				k(x,y)=\int_{(c,\med\{x,y,c\})}\frac{dt}{\rho(t)}
				=\Bigl|\int_c^{\med\{x,y,c\}}\frac{dt}{\rho(t)}\Bigr|.
				\]
		\item Let $\chi_E$ denote the characteristic function of a set $E$.
				Then, $H_{c,\rho}(I)$ is the integral transform of the Hilbert space $L_\rho^2(I)$
				via the mapping
				\[
				\phi\colon x\in I\mapsto\frac{\chi_{(c,x)}}{\rho}\in L_\rho^2(I),
				\]
				that is, $H_{c,\rho}(I)=\hat\phi(L_\rho^2(I))$.
		\item The integral transform $\hat\phi\colon L_\rho^2(I)\to H_{c,\rho}(I)$ is
				an isometry.
				In particular, the family of functions $\phi(I)$ is complete in $L_\rho^2(I)$.
	\end{enumerate}
\end{thm}

\begin{proof}
	(i)
	Let $k_y=k(\cdot,y)$.
	From~\eqref{eq:rho^-1}, it is clear that $k_y\in H_{c,\rho}(I)$.
	We verify the reproducing property of $k_y$.
	We consider cases based on the relationship between $y$ and $c$.
	If $y\le c$,
	\begin{align}
		k(x,y)
		&=\begin{cases}
			\int_{x\vee y}^c\rho^{-1}dt, & (x\le c), \\
			0, & (c\le x).
		\end{cases}
	\end{align}
	Hence, for $f\in H_{c,\rho}(I)$,
	\[
		\inner{f,k_y}=\int_If'\ol{k_y'}\rho\,dt
		=\int_y^cf'\cdot(-\rho^{-1})\rho\,dt=\int_c^yf'(t)\,dt=f(y).
	\]
	The case $c\le y$ is handled similarly.

	(ii)
	From assumption~\eqref{eq:rho^-1}, the mapping $\phi$ is well-defined.
	The reproducing kernel of the integral transform $\hat\phi(L_\rho^2(I))$ is
	given by~\eqref{eq:phi-kernel}:
	\begin{align*}
		\inner{\phi(y),\phi(x)}_{L_\rho^2(I)}
		&=\int_I\frac{\chi_{(c,x)\cap(c,y)}}{\rho}\,dt
		=\int_{(c,\med\{x,y,c\})}\frac{dt}{\rho}.
	\end{align*}
	This coincides with the reproducing kernel of $H_{c,\rho}(I)$.
	Hence, $\hat\phi(L_\rho^2(I))=H_{c,\rho}(I)$.
	
	(iii)
	For $f\in L_\rho^2(I)$, the integral transform $\hat f=\hat\phi(f)$ is, by definition,
	\[
		\hat f(x)=\int_If\frac{\chi_{(c,x)}}{\rho}\rho\,dt
		=\begin{cases}
			-\int_c^xf\,dt, & (x\le c), \\
			\int_c^xf\,dt, & (x\ge c).
		\end{cases}
	\]
	Hence, $|\hat f'(x)|=|f(x)|$ a.e. $x\in I$, and
	\[
	\|\hat f\|_{H_{c,\rho}(I)}^2
	=\int_I|\hat f'|^2\rho\,dt
	=\int_I|f|^2\rho\,dt=\|f\|_{L_\rho^2(I)}^2.
	\]
	Therefore, the integral transform $\hat\phi$ is an isometry,
	and by Theorem~\ref{th:int-trans}, the set $\phi(I)$ is complete in $L_\rho^2(I)$.
\end{proof}

\subsection{Transformation Sequences and Tensor Products}

To handle the Fourier transform on $\R^N$ using the inversion theorem,
it is necessary to identify a manageable RKHS isometric to the $N$-dimensional $L^2$ space.
Here, we aim to construct the desired higher-dimensional RKHS
through the tensor product of the one-dimensional RKHS introduced in the previous section.

The tensor product Hilbert space $\otimes_{i=1}^n\calH_i$ of Hilbert spaces
$\calH_i$ ($i=1,\dots,n$) is the completion of the algebraic tensor product
$\odot_{i=1}^n\calH_i$, equipped with an inner product satisfying
\begin{equation}\label{eq:tensorp}
	\inner{\otimes_{i=1}^na_i,\otimes_{i=1}^nb_i}_{\otimes_{i=1}^n\calH_i}
	=\prod_{i=1}^n\inner{a_i,b_i}_{\calH_i}
\end{equation}
for any $a_i,b_i\in\calH_i$ ($i=1,\dots,n$).

Additionally, the tensor product
$\otimes_{i=1}^nS_i\colon\otimes_{i=1}^n\calH_i\to\otimes_{i=1}^n\calK_i$
of bounded operators $S_i\colon\calH_i\to\calK_i$ ($i=1,\dots,n$)
between Hilbert spaces is the unique bounded operator satisfying
\begin{equation}\label{eq:OpTensor}
	(\otimes_{i=1}^nS_i)(\otimes_{i=1}^na_i)=\otimes_{i=1}^nS_ia_i,
\end{equation}
for $a_i\in\calH_i$ ($i=1,\dots,n$) (see, e.g.,~\cite{ReedSimon72,PaulsenRaghupathi16}).

The following result regarding the tensor product of RKHSs is well known.

\begin{prp}[\cite{PaulsenRaghupathi16}*{p.~73}]\label{prp:tensorRKHS}
	Let $X$ be an RKHS on $E$ and $Y$ an RKHS on $F$.
	The tensor product Hilbert space $X\otimes Y$ is an RKHS on $E\times F$.
	If $k_x^1$ is the reproducing kernel of $X$ for $x\in E$ and $k_y^2$ is
	the reproducing kernel of $Y$ for $y\in F$,
	then the reproducing kernel of $X\otimes Y$ for $(x,y)\in E\times F$ is given by
	$k_x^1\otimes k_y^2$.
\end{prp}

Next, we describe the relationship between tensor product Hilbert spaces and integral transforms. In what follows, the reproducing kernel of the integral transform $\hat\phi(\calH)$
for a point $x\in E$, defined by a map $\phi\colon E\to\calH$, is denoted by $k[\phi]_x$.

\begin{prp} \label{prp:IT-tens}
	Let $\calH_i$ be Hilbert spaces, $E_i$ be sets, and $\phi_i\colon E_i\to\calH_i$
	be maps ($i=1,\dots,n$).
	Then, the integral transform defined by the map
	$\otimes_{i=1}^n\phi_i\colon\prod_{i=1}^nE_i\to\otimes_{i=1}^n\calH_i$
	such that $(\otimes_{i=1}^n\phi_i)(x_1,\dots,x_n)=\otimes_{i=1}^n\phi_i(x_i)$ satisfies:
	\begin{enumerate}
		\item For any $f_i\in\calH_i$ ($i=1,\dots,n$),
			$\what{\otimes_{i=1}^n\phi_i}(\otimes_{i=1}^nf_i)
				=\otimes_{i=1}^n\hat\phi_if_i$.
		\item For $x=(x_i)\in\prod_{i=1}^nE_i$, $k[\otimes_{i=1}^n\phi_i]_x
			=\otimes_{i=1}^nk[\phi_i]_{x_i}$.
			In particular, as RKHSs on $\prod_{i=1}^nE_i$,
			\[
				\widehat{\otimes_{i=1}^n\phi_i}\Bigl(\otimes_{i=1}^n\calH_i\Bigr)
				=\otimes_{i=1}^n\hat\phi_i(\calH_i).
			\]
		\item If $\phi_i(E_i)$ is complete in $\calH_i$ ($i=1,\dots,n$), then
			$(\otimes_{i=1}^n\phi_i)(\prod_{i=1}^nE_i)$ is also complete in $\otimes_{i=1}^n\calH_i$.
	\end{enumerate}
\end{prp}

\begin{proof}
	First, note that both the integral transform
	$\what{\otimes_{i=1}^n\phi_i}(\otimes_{i=1}^n\calH_i)$ and the tensor product Hilbert space
	$\otimes_{i=1}^n\hat\phi_i\calH_i$ are RKHSs on $\prod_{i=1}^nE_i$.
	For $x_i\in E_i$ and $f_i\in\calH_i$ ($i=1,\dots,n$),
	(i) follows from a straightforward calculation based on the definition.
	To prove (ii), using~\eqref{eq:phi-kernel},
	the reproducing kernel of $\otimes_{i=1}^n\calH_i$ for $y=(y_i)$ is
	\begin{align*}
		k[\otimes_{i=1}^n\phi_i]_x(y)
		&=\inner{\otimes_{i=1}^n\phi_i(y),\otimes_{i=1}^n\phi_i(x)}_{\otimes_{i=1}^n\calH_i}
			=\prod_{i=1}^n\inner{\phi_i(y_i),\phi_i(x_i)}_{\calH_i} \\
		&=\prod_{i=1}^n k[\phi_i]_{x_i}(y_i)=(\otimes_{i=1}^n k[\phi_i]_{x_i})(y).
	\end{align*}
	Hence, $k[\otimes_{i=1}^n\phi_i]_x=\otimes_{i=1}^nk[\phi_i]_{x_i}$.
	By Proposition~\ref{prp:tensorRKHS}, the left-hand side and right-hand side represent
	the reproducing kernels of the RKHS $\widehat{\otimes_{i=1}^n\phi_i}(\otimes_{i=1}^n \calH_i)$
	and the RKHS $\otimes_{i=1}^n\hat\phi_i(\calH_i)$ for $x$, respectively.
	The equality of reproducing kernels implies that these RKHSs are identical.
	(iii) follows from the fact that the span of
	$(\otimes_{i=1}^n\phi_i)(\prod_{i=1}^nE_i)=\otimes_{i=1}^n(\phi_i(E_i))$ is dense
	in the algebraic tensor product $\odot_{i=1}^n\calH_i$,
	which is straightforward to verify.
\end{proof}

Transformation sequences are closed under tensor products.

\begin{prp}\label{prp:EHW}
	If $E\oset{\phi_i}\to\calH_i\oset{S_i}\to\calW_i$ are transformation sequences on $T_i$ ($i=1,\dots,N$), then
	\[
		\begin{tikzcd}
			\prod_{i=1}^NE_i\ar[r,"\otimes_{i=1}^N\phi_i"] 
			& \otimes_{i=1}^N\calH_i\ar[r,"\otimes_{i=1}^NS_i"]
			& \otimes_{i=1}^N\calW_i
		\end{tikzcd}
	\]
	is a transformation sequence on $\prod_{i=1}^NT_i$.
\end{prp}

\begin{proof}
	It suffices to show that the operator $\otimes_{i=1}^NS_i$ is an isometry. This can be demonstrated by showing isometry for simple tensors, which follows directly
	from~\eqref{eq:tensorp} and the definition~\eqref{eq:OpTensor}.
\end{proof}

\section{Application}

\subsection{Paley-Wiener Space}\label{sec:PW}

As preparation for treating Fourier transforms using the theory of reproducing kernels,
we consider the Paley-Wiener space.
For $a>0$, let $I_a=(-a,a)$.
We define a mapping $\phi\colon\C\to L^2(I_a)$ by
\begin{equation}\label{eq:e-itx}
    \phi(t,x)=\phi(x)(t)=\frac1{\sqrt{2\pi}}e^{it\ol x},\quad(t,x)\in I_a\times\C.
\end{equation}
The set $\phi(\C)$ is complete in $L^2(I_a)$, because the family $\{e^{i\pi nt/a}\}_{n\in\Z}$
is a complete orthogonal system in $I_a$.
Therefore, by Theorem~\ref{th:int-trans}, the integral transform $\hat\phi\colon L^2(I_a)\to\hat\phi(L^2(I_a))$ is an isometry.
The integral transform induced by $\phi$ is related to the Fourier transform $\calF$
on $L^2(\R)$ as follows: for $f\in L^2(\R)$,
\[
\hat\phi(f|_{I_a})=\frac1{\sqrt{2\pi}}\int_{I_a}f(t)e^{-it\cdot}\,dt
=\calF(f\chi_{I_a}).
\]
The image $\hat\phi(L^2(I_a))$ of the integral transform, denoted by $PW(a)$,
is an RKHS on $\C$, called the \emph{Paley-Wiener space of order $a$}.
The reproducing kernel $K_a(x,y)$ of $PW(a)$ is given by, for $(x,y)\in\C^2$,
\begin{align}\label{eq:DirichKern}
	K_a(x,y)
	&=\inner{\phi(y),\phi(x)}_{L^2(I_a)}
		=\frac1{2\pi}\int_{I_a}e^{it\ol y}\ol{e^{it\ol x}}\,dt
		=\frac{\sin(a(x-\ol y))}{\pi(x-\ol y)}.
\end{align}
Since the integral kernel $\ol{\phi(x)}$ of $PW(a)$ is an entire function of $x$,
it is easy to see that differentiation and integration can be interchanged.
This implies that elements of $PW(a)$ are entire functions.
The uniqueness theorem for analytic functions implies that
if an entire function vanishes on $\R$, then it is identically zero on $\C$.
Thus, the real line $\R$ is a uniqueness set for $PW(a)$.
The following lemma provides the equality~\eqref{eq:k(x,y)iso} for the norm in $PW(a)$,
which is essential for our purposes.

\begin{lem}\label{lem:intsinsin}
	For all $x,y\in\R$,
	\[
	\int_\R\frac{\sin(a(t-x))\sin(a(t-y))}{(t-x)(t-y)}\,dt
	=\pi\frac{\sin(a(y-x))}{y-x}.
	\]
\end{lem}

\begin{proof}
    Using partial fraction decomposition, we have
    \begin{align}
        \int_\R\frac{\sin(a(t-x))\sin(a(t-y))}{(t-x)(t-y)}\,dt
        &=\frac1{y-x}\Bigl\{\int_{-\infty}^\infty\frac{\sin(a(t-x))\sin(a(t-y))}{t-y}\,dt \\
        &\qquad\qquad-\int_{-\infty}^\infty\frac{\sin(a(t-x))\sin(a(t-y))}{t-x}\,dt\Bigr\} \\
        &=\frac{F(x,y)-F(y,x)}{y-x}. \label{eq:F-F}
    \end{align}
    Here, the two integrals on the right-hand side are improper integrals,
		and the first one is denoted by $F(x,y)$.
    To compute $F(x,y)$, we make the substitution $u=t-y$:
    \begin{align}
        F(x,y)
        &=\int_{-\infty}^\infty\frac{\sin(a(t-x))\sin(a(t-y))}{t-y}\,dt \\
        &=\int_{-\infty}^\infty\frac{\sin(a(t+y-x))\sin(at)}{t}\,dt \\
        &=\lim_\ninf\int_{-n}^n\frac{\cos(at)\sin(a(y-x))\sin(at)}{t}\,dt \\
        &=\frac{\sin(a(y-x))}2\int_{-\infty}^\infty\frac{\sin(2at)}t\,dt
            =\sin(a(y-x))\int_0^\infty\frac{\sin t}t\,dt \\
        &=\frac\pi2\sin(a(y-x)).
    \end{align}
    Substituting this into~\eqref{eq:F-F}, we obtain the desired result.
\end{proof}

In Lemma~\ref{lem:iso-met} let $T$ be the restriction map to $\R$, and set $H=PW(a)$, $E=\C$, $F=\R$, and $K=L^2(\R)$.
Lemma~\ref{lem:intsinsin} implies that the norm of $PW(a)$ is given by the $L^2(\R)$ norm
(cf.~\cite{Saitoh97}*{p.~62}): for $f\in PW(a)$,
\begin{equation}\label{eq:PWnorm}
  \|f\|_{PW(a)}^2=\int_\R|f(x)|^2\,dx.
\end{equation}

\subsection{Plancherel's Theorem}

In this section, we provide an alternative proof of Plancherel's theorem using integral transforms and the inversion theorem.
This approach offers a different perspective compared to standard proofs that
often rely on approximation by smooth functions.
We use the notation from the previous section.
For $t=(t_i), x=(x_i)\in\C^N$, let $t\cdot x=\sum_{i=1}^Nt_i\ol{x}_i$.

\begin{thm}[Plancherel's Theorem]
	For $f\in L^1(\R^N)$ ($N\in\N$), the Fourier transform $\calF f$
	and the conjugate Fourier transform $\calF^* f$ are defined for $t\in\R^N$ by
	\begin{align}
		\calF f(t)
		&= \frac1{(2\pi)^{N/2}}\int_{\R^N}f(x)e^{-it\cdot x}\,dx, \\
		\calF^* f(t)
		&= \frac1{(2\pi)^{N/2}}\int_{\R^N}f(x)e^{it\cdot x}\,dx.
	\end{align}
	Then, the restrictions of $\calF$ and its conjugate $\calF^*$ to
	$L^1(\R^N)\cap L^2(\R^N)$ extend to unitary operators on $L^2(\R^N)$,
	and they are inverses of each other.
\end{thm}

\begin{proof}
	The integral transform $\hat\phi\colon L^2(I_a)\to PW(a)$ is a Fourier transform.
	Since $\hat\phi$ is an isometry, its extension
	$\hat\phi\colon L^2(I_a)\to L^2(\R)$ is also an isometry.
	As $L^2(I_a)^{\otimes N} \cong L^2(I_a^N)$,
	the operator $\hat\phi^{\otimes N}: L^2(I_a)^{\otimes N} \to PW(a)^{\otimes N}$
	can be identified with the Fourier transform on $L^2(I_a^N)$.
	By Proposition~\ref{prp:EHW}, the tensor product of isometries is an isometry.
	Therefore, $\hat\phi^{\otimes N}$ is an isometry.
	The Fourier transform on $L^2(\R^N)$ is defined as the mean convergence
	of the Fourier transform on the Cartesian products $I_a^N$ as $a \to \infty$.
	From the above, it follows that the Fourier transform $\calF: L^2(\R^N) \to L^2(\R^N)$
	is an isometry.
	Similarly, the conjugate Fourier transform $\calF^*: L^2(\R^N) \to L^2(\R^N)$
	is also an isometry.

	Next, we show that $\calF$ and $\calF^*$ are mutual inverses, i.e.,
	$\calF\calF^*=\calF^*\calF=I$, where $I$ denotes the identity.
	Let $\calW(I_a)=H_{0,1}(I_a)$.
	Then $\calW(I_a)$ is an RKHS on $I_a$, and the indefinite integral operator
	$S\colon f\in L^2(I_a)\mapsto\int_0^\blt f(x)\, dx\in\calW(I_a)$ is an isometry.
	Considering the transformation sequence on $I_a$,
	\[
	\begin{tikzcd}
		I_a \ar[r, "\phi"] & L^2(I_a) \ar[r, "S"] & \calW(I_a)
	\end{tikzcd}
	\]
	and taking the $N$-fold tensor product and applying Proposition~\ref{prp:EHW},
	we obtain the following transformation sequence on $I_a^N$:
	\[
	\begin{tikzcd}
		I_a^N \ar[r, "\phi^{\otimes N}"] & L^2(I_a)^{\otimes N} \ar[r, "S^{\otimes N}"]
		& \calW(I_a)^{\otimes N}
	\end{tikzcd}.
	\]
	Since the tensor product Hilbert space $L^2(I_a)^{\otimes N}$
	is isomorphic to $L^2(I_a^N)$,
	we have, for $t=(t_j),\ x=(x_j)\in\R^N$,
	\[
		\phi^{\otimes N}(t,x)=\frac1{(2\pi)^{N/2}}e^{it\cdot x}.
	\]

	On the other hand, by the definition of the tensor product of maps and Fubini's theorem,
	for $I(t)=\prod_{j=1}^N (0, t_j)\ (\subset I_a^N)$ and $f_1,\dots,f_N\in L^2(I_a)$,
	we have
	\[
		S^{\otimes N}(\otimes_{j=1}^N f_j)(t)
		= \prod_{j=1}^N Sf_j(t_j)
		= \prod_{j=1}^N \int_0^{t_j} f_j(s_j)\,ds_j
		= \int_{I(t)} (\otimes_{j=1}^N f_j)(s)\,ds.
	\]
	Since the span of simple tensors $\otimes_{j=1}^N f_j$ is dense in $L^2(I_a^N)$,
	it follows that for all $f\in L^2(I_a^N)$,
	\[
		(S^{\otimes N}f)(t) = \int_{I(t)}f(s)\,ds,
	\]
	i.e., $S^{\otimes N}$ is the $N$-dimensional indefinite integral operator.
	Similarly, the norm of $PW(a)^{\otimes N}$ which is an RKHS on $\C^N$ satisfies
	\[
		\|f\|_{PW(a)^{\otimes N}}^2 = \int_{\R^N}|f(x)|^2dx
	\]
	for $f\in PW(a)^{\otimes N}$.
	Therefore, when $f\in L^2(\R^N)$ and $f|_{I_a^N}\in L^2(I_a^N)$,
	applying the inverse transformation formula~\eqref{eq:invtrans2},
	we obtain, for $t=(t_j)\in I_a^N$, $x=(x_j)\in\R^N$,
	\begin{equation} \label{eq:FourRN}
		\int_{I(t)} f(s)\,ds
		=\frac{1}{(2\pi)^{N/2}}\int_{\R^N}\calF(f\chi_{I_a^N})(x)\int_{I(t)}e^{is\cdot x}\,ds\,dx.
	\end{equation}
	By Theorem~\ref{thm:invtrans}, for fixed $t\in I_a^N$, we have
	$\int_{I(t)}e^{is \cdot x}\,ds\in L^2(\R^N)$.
	Since the Fourier transform is an isometry,
	$\calF(f\chi_{I_a^N})\to\calF(f)$ ($a\to\infty$) in $L^2(\R^N)$.
	Thus,	letting $a\to\infty$ in~\eqref{eq:FourRN},
	we obtain, for all $f\in L^2(\R^N)$ and $t\in\R^N$,
	\begin{align}\label{eq:FourRNafterlimit}
		\int_{I(t)} f(x)\,dx
		&=\frac{1}{(2\pi)^{N/2}}\int_{\R^N}\calF(f)(x)\int_{I(t)}e^{is\cdot x}\,ds\,dx.
	\end{align}
	If, in addition, $\calF(f)\in L^1(\R^N)$, then, since $|e^{is\cdot x}|=1$,
	we can apply Fubini's theorem to interchange the order of integration in
	\eqref{eq:FourRNafterlimit}, obtaining
	\[
		\int_{I(t)} f(x) \, dx = \int_{I(t)} \calF^* \calF(f)(s) \, ds.
	\]
	Furthermore, if we write $x=(x_1,x')$ and $t=(t_1,t')\in\R^N$,
	then by Fubini's theorem we have,
	\[
		\int_0^{t_1}\int_{I(t')}f(x_1, x')\,dx'\,dx_1
		=\int_0^{t_1}\int_{I(t')}\calF^*\calF(f)(x')\, dx'\,dx_1.
	\]
	Differentiating both sides with respect to $t_1$ at $t_1=x_1$,
	we obtain, for almost every $x'\in\R^{N-1}$,
	\[
		\int_{I(t')}f(x_1,x')\,dx'=\int_{I(t')}\calF^*\calF(f)(x')\,dx'.
	\]
	Repeating this differentiation $N-1$ times with respect to $t_2,\dots,t_N$
	at $t_2=x_2,\dots,t_N=x_N$,
	we conclude that $f(x)=\calF^*\calF(f)(x)$ for almost every $x\in\R^N$.
	Now, if $f\in C_c^\infty(\R^N)$, then by integration by parts
	we obtain the identity
	\[
		\calF\Bigl(\frac{\bd^{2N}f}{\bd x_1^2\cdots\bd x_N^2}\Bigr)
		=(-1)^N\prod_{i=1}^Nx_i^2\cdot\calF(f).
	\]
	Since the left-hand side is bounded and since $\calF(f)\in L^2(\R^N)$,
	it is easy to see that $\calF(f)\in L^1(\R^N)$.
	Since $C_c^\infty(\R^N)$ is dense in $L^2(\R^N)$,
	we conclude that the set of $f$ with $\calF(f)\in L^1(\R^N)$ is dense in $L^2(\R^N)$.
	Therefore, by the boundedness of $\calF$ and $\calF^*$, the desired identity
	$\calF^*\calF=I$ holds by continuity.
	Similarly, it can be shown that $\calF\calF^*=I$.
	Thus, Plancherel's theorem is demonstrated.
\end{proof}


\end{document}